\documentclass[12pt]{amsart}
\usepackage{amsmath,amssymb,amsthm,amscd}
\newtheorem{thm}{Theorem}
\newtheorem*{thmnonum}{Theorem}
\newtheorem{lem}[thm]{Lemma}

\newtheorem{cor}[thm]{Corollary}
\newtheorem*{rem}{Remark}

\newcommand{\SL}{{\rm SL}}

\newcommand{\C}{\mathbb{C}}
\renewcommand{\H}{\mathbb{H}}
\newcommand{\Q}{\mathbb{Q}}
\newcommand{\Z}{\mathbb{Z}}

\newcommand{\legen}[2]{\genfrac{(}{)}{}{}{#1}{#2}}
\newcommand{\Gal}{{\rm Gal}}
\newcommand{\ord}{{\rm ord}}

\newcommand{\Div}{\operatorname{div}}
\newcommand{\DIV}{\operatorname{Div}}
\newcommand{\lcm}{\operatorname{lcm}}

\begin{document}

\title[Eta-quotients]{On spaces of modular forms spanned by eta-quotients}
\author{Jeremy Rouse}
\address{Department of Mathematics, Wake Forest University,
  Winston-Salem, NC 27109}
\email{rouseja@wfu.edu}
\author{John J. Webb}
\address{Department of Mathematics, Wake Forest University,
  Winston-Salem, NC 27109 \\ 
  Department of Mathematics \& Statistics, James Madison University, Harrisonburg, VA 22802}
\email{webbjj@jmu.edu}

\begin{abstract}
An eta-quotient of level $N$ is a modular form of the shape $f(z) = \prod_{\delta | N} \eta(\delta z)^{r_{\delta}}$. We study the problem of determining levels $N$
for which the graded ring of holomorphic modular forms for $\Gamma_{0}(N)$
is generated by (holomorphic, respectively weakly holomorphic) eta-quotients
of level $N$. In addition, we prove that if $f(z)$ is a holomorphic modular
form that is non-vanishing on the upper half plane and has integer Fourier
coefficients at infinity, then $f(z)$ is an integer multiple of an eta-quotient.
Finally, we use our results to determine the structure of the cuspidal
subgroup of $J_{0}(2^{k})(\Q)$.
\end{abstract}

\maketitle

\section{Introduction}
\label{intro}

In 1877, Richard Dedekind \cite{Dedekind} defined the function
\begin{equation}\label{eta-def}
  \eta(z) = q^{1/24} \prod_{n=1}^{\infty} (1-q^{n})\, , \quad q = e^{2 \pi i z}\, ,
\end{equation}
for $z \in \H = \{ x + iy \in \C : y > 0 \}$, which is now known as the
Dedekind eta-function. The reciprocal of the Dedekind eta-function is
\[
  \frac{1}{\eta(z)} = q^{-1/24} \sum_{n=0}^{\infty} p(n) q^{n} \, ,
\]
where $p(n)$ is the number of partitions of $n$. The modular transformations
of the Dedekind eta-function, which take the form
\[
  \eta\left(\frac{az+b}{cz+d}\right) = \epsilon(a,b,c,d) (cz+d)^{1/2} \eta(z)
  \text{ for } \left[ \begin{matrix} a & b \\ c & d \end{matrix} \right]
  \in \SL_{2}(\Z) \text{ and } \epsilon(a,b,c,d)^{24} = 1,
\]
play an important role in Hardy and Ramanujan's exact formula for $p(n)$, derived using the circle method.

An \emph{eta-quotient} of level $N$ is a function of the shape
\[
  f(z) = \prod_{\delta | N} \eta(\delta z)^{r_{\delta}},
\]
where $r_{\delta} \in \Z$. As a consequence of the product definition for
$\eta(z)$, any eta-quotient is non-vanishing on $\H$. The following are
examples of well-known eta-quotients:
\[
  \Delta(z) = \eta(z)^{24} = \sum_{n=1}^{\infty} \tau(n) q^{n}\, , \quad
  \frac{\eta(2z)^{5}}{\eta(z)^{2} \eta(4z)^{2}} = \sum_{n \in \Z} q^{n^{2}}\, , \quad
  \frac{\eta(4z)^{8}}{\eta(2z)^{4}} = \sum_{n=0}^{\infty} \sigma_1(2n+1) q^{2n+1},
\]
where for all positive integers $j$, $t$, we have $\sigma_t(j) = \sum_{d\mid j } d^t$.

Let 
\[
  \Gamma_{0}(N) = \left\{ \left[ \begin{matrix} a & b \\ c & d \end{matrix} \right]
  : a, b, c, d \in \Z, ad - bc = 1 \text{ and } N | c \right\}.
\]
A weight $k$ weakly holomorphic modular form is a function on $\H$ that obeys the weight $k$ modular transformation law for $\Gamma_0(N)$, is holomorphic on $\H$, but may possess poles at the cusps.  In the 1950s, Morris Newman (see \cite{Newman1} and
\cite{Newman2}) used the Dedekind eta-function to systematically build
weakly holomorphic modular forms for $\Gamma_0(N)$.
A key ingredient in his work is a classification of when $f(z)$ is
actually a modular form for $\Gamma_{0}(N)$. If
\begin{align}\label{congruenceconditions}
  & \sum_{\delta | N} \delta r_{\delta} \equiv 0 \pmod{24}, \quad
  \sum_{\delta | N} \frac{N}{\delta} r_{\delta} \equiv 0 \pmod{24}, \text{ and }\\ \notag
  & \prod_{\delta | N} \delta^{r_{\delta}} \text{ is the square of a rational
number,}
\end{align}
then $f(z)$ transforms like a weight $k = \frac{1}{2} \sum_{\delta | N} r_{\delta}$
modular form for $\Gamma_{0}(N)$. We denote the set of weight $k$ modular forms
on $\Gamma_{0}(N)$ by $M_{k}(\Gamma_{0}(N))$.

Let $Y_{0}(N)$ be the algebraic curve $\H/\Gamma_{0}(N)$, and let
$X_{0}(N)$ be its compactification. This algebraic curve has a model
defined over $\Q$ given by the classical modular equation
$\Phi_{N}(x,y) = 0$ (see for example \cite{Shimura}, pg. 109-110). In
\cite{Ligozat}, Ligozat determines the order of vanishing of an eta
quotient at the cusps of the level $N$ modular curve $X_{0}(N)$.

\begin{thmnonum}
Let $c$, $d$ and $N$ be positive integers with $d | N$ and $\gcd(c,d) = 1$.
If $f(z)$ is an eta-quotient, then the order of vanishing of $f(z)$ at
the cusp $c/d$ is
\[
  \frac{N}{24} \sum_{\delta | N} \frac{\gcd(d,\delta)^{2} r_{\delta}}{\gcd(d,N/d) d \delta}.
\]
\end{thmnonum}

\begin{rem}
The formula for the order of vanishing only depends on $d$,
and not on $c$.
\end{rem}

In \cite{Ono}, Ken Ono observes that every holomorphic modular form
for $\SL_{2}(\Z)$ can be expressed as a linear combination of eta-quotients
of level $4$, and poses the following problem: ``Classify the spaces of
modular forms which are generated by eta-quotients.'' The goals of the
present paper are to address this question, to give an intrinsic characterization of level $N$ eta-quotients, and to apply this information to the study
of the cuspidal subgroup of the Jacobian of $X_{0}(N)$.

\section{Generating spaces with holomorphic and weakly holomorphic eta-quotients}

Combining the hypotheses in the theorems of Newman and Ligozat show that
the sequences of integers $\{ r_{\delta} : \delta | N\}$,
which are the exponents of an eta-quotient in $M_{k}(\Gamma_{0}(N))$, correspond
to integer points inside a $d(N) - 1$-dimensional convex polytope, where $d(N)$ is the number of divisors of $N$.

{\bf Example.}
Using the lattice point enumeration program LattE \cite{LattE}, we are able to
determine precisely the number of eta-quotients in $M_{k}(\Gamma_{0}(36))$
for various values of $k$. These eta-quotients correspond to lattice points in
an $8$-dimensional polytope whose volume increases quickly with $k$. There are
4,988 eta-quotients in $M_{2}(\Gamma_{0}(36))$ and there are 703,060,312
eta-quotients in $M_{12}(\Gamma_{0}(36))$.

\begin{thm}
\label{hologen}
There are precisely $121$ positive integers $N \leq 500$ so that the
graded ring of modular forms for $\Gamma_{0}(N)$ is generated by
eta-quotients.
\end{thm}

A necessary condition for eta-quotients to generate the graded ring of
modular forms on $\Gamma_0(N)$ if $N$ is composite is for
$M_2(\Gamma_0(N))$ to be spanned by eta-quotients.  Numerical evidence
suggests that it is only possible to find enough eta-quotients to span
$M_{2}(\Gamma_{0}(N))$ if $N$ is sufficiently composite. We make this
precise in our next result.

\begin{thm}
\label{expobound}
Suppose that $f(z) = \prod_{\delta | N} \eta(\delta z)^{r_{\delta}} \in M_{k}(\Gamma_{0}(N))$. Then we have
\[
  \sum_{\delta | N} |r_{\delta}| \leq 2k \prod_{p | N} \left(\frac{p+1}{p-1}\right)^{\min \{2, \ord_{p}(N)\}}.
\]
\end{thm}

\begin{rem}
The bound is sharp. For each $N$, there is a weight $k$ and an eta-quotient $f(z) \in M_{k}(\Gamma_{0}(N))$ for which the bound is achieved.
\end{rem}

If $k$ and the number of divisors of $N$ are fixed, the number of
sequences of integers $r_{\delta}$ that satisfy the inequality in
Theorem~\ref{expobound} is bounded. Since $\dim M_{k}(\Gamma_{0}(N))$
tends to infinity with $N$, it follows that for a fixed $k$ and a
given number of prime divisors of $N$, there are only finitely many
spaces $M_{k}(\Gamma_{0}(N))$ that are spanned by
eta-quotients. Accordingly, a space $M_{k}(\Gamma_{0}(N))$ will be
spanned by eta-quotients only if $N$ is ``sufficiently composite'' in
relation to its size.

\begin{rem}
Theorem~\ref{expobound} and the consequences regarding
$M_{k}(\Gamma_{0}(N))$ being spanned by eta-quotients were first
obtained by Soumya Bhattacharya in 2011 who is (as of this writing) a
Ph.D. student of Professor Don Zagier at the University of Bonn. 
Bhattacharya's work is not yet available in manuscript form.
\end{rem}

Our next result gives examples of spaces that are not spanned by eta-quotients.

\begin{cor}
\label{primelevel}
If $p$ is prime, then $M_{2}(\Gamma_{0}(4p))$ is spanned by eta-quotients
if and only if $p \leq 13$.
\end{cor}

Instead of insisting that every element of $M_{k}(\Gamma_{0}(N))$ be
expressible as a linear combination of eta-quotients which are
holomorphic on $\mathbb{H}$, we could instead ask if every element of
$M_{k}(\Gamma_{0}(N))$ is expressible as a linear combination of
\emph{weakly holomorphic} eta-quotients of weight $k$ and level $N$
with a pole only at infinity.

For example, $M_{2}(\Gamma_{0}(22))$ has dimension $5$, but only
contains $4$ eta-quotients, namely
\[
  \frac{\eta(2z)^{4} \eta(22z)^{4}}{\eta(z)^{2} \eta(11z)^{2}}, \eta(z)^{2} \eta(11z)^{2}, \eta(2z)^{2} \eta(22z)^{2}, \text{ and }
  \frac{\eta(z)^{4} \eta(11z)^{4}}{\eta(2z)^{2} \eta(22z)^{2}}.
\]
The basis element $f(z) = q^{4} + q^{6} + q^{8} + q^{10} - q^{11} +
\cdots$ is not expressible as a linear combination of these
four. However, if $g(z) = \frac{\eta(22z)^{22} \eta(z)}{\eta(2z)^{2}
  \eta(11z)^{11}}$, then $f(z) g(z) \in M_{12}(\Gamma_{0}(22))$ and
every holomorphic modular form in $M_{12}(\Gamma_{0}(22))$ is a linear
combination of (holomorphic) eta-quotients. Since $g(z)$ is
non-vanishing at all cusps except infinity, this implies that $f(z)$
is a linear combination of weakly holomorphic eta-quotients with poles
only at infinity.

Next, we will study the levels $N$ for which every form in
$M_{k}(\Gamma_{0}(N))$ is a linear combination of weakly holomorphic
eta-quotients with a pole only at infinity. To state our result, let
$\mathcal{R}_{k}(\Gamma_{0}(N))$ be the vector space of
all weakly holomorphic modular forms that have a pole only at
the cusp at infinity. Let $\mathcal{E}_{k}(\Gamma_{0}(N))$ be the subspace
of $\mathcal{R}_{k}(\Gamma_{0}(N))$ consisting of forms that are linear
combinations of eta-quotients with a pole only at infinity. Note
that $\mathcal{R}_{0}(\Gamma_{0}(N))$ and $\mathcal{E}_{0}(\Gamma_{0}(N))$
are rings, and $\mathcal{R}_{k}(\Gamma_{0}(N))$ and
$\mathcal{E}_{k}(\Gamma_{0}(N))$ naturally have the structure of modules
over $\mathcal{R}_{0}(\Gamma_{0}(N))$ and $\mathcal{E}_{0}(\Gamma_{0}(N))$, respectively.

Our next result is the following. Let $\epsilon_{2}(\Gamma_{0}(N))$ and
$\epsilon_{3}(\Gamma_{0}(N))$ denote the number of orbits of elliptic points
of order $2$ and $3$ for $\Gamma_{0}(N)$, respectively.

\begin{thm}
\label{finitecodim}
\begin{enumerate}
\item Suppose that $N$ is composite or $N \in \{ 2, 3, 5, 7, 13 \}$.
Suppose also that either $\epsilon_{3}(\Gamma_{0}(N)) = 0$
or $6 | k$, and either $\epsilon_{2}(\Gamma_{0}(N)) = 0$ or $4 | k$.
Then $\mathcal{E}_{k}(\Gamma_{0}(N))$ has finite codimension in $\mathcal{R}_{k}(\Gamma_{0}(N))$.
\item If $N$ is prime and $N = 11$ or $N \geq 17$, then
$\mathcal{E}_{k}(\Gamma_{0}(N))$ has infinite codimension in $\mathcal{R}_{k}(\Gamma_{0}(N))$ for all non-negative even integers $k$.
\end{enumerate}
\end{thm}

Our next result gives a classification of the levels $N$ for which
every element of $M_{k}(\Gamma_{0}(N))$ has an expression in terms of
weakly holomorphic eta-quotients.

\begin{thm}
\label{weakthm}
Let $N$ be a positive integer. Then the following are equivalent.
\begin{enumerate}
\item For all positive even integers $k$, $M_k(\Gamma_0(N)) \subset \mathcal{R}_k(\Gamma_0(N))$.  In other words, every element of
$M_{k}(\Gamma_{0}(N))$ is expressible as a linear combination of
weakly holomorphic weight $k$ eta-quotients of level $N$ with poles
only at infinity.
\item $\mathcal{R}_{k}(\Gamma_{0}(N)) = \mathcal{E}_{k}(\Gamma_{0}(N))$
for all non-negative even integers $k$.
\item $\mathcal{R}_{2}(\Gamma_{0}(N)) = \mathcal{E}_{2}(\Gamma_{0}(N))$.
\end{enumerate}
\end{thm}

Theorem~\ref{finitecodim} and the proof of Theorem~\ref{weakthm}
show that in order for $\mathcal{R}_{2}(\Gamma_{0}(N))$ to
equal $\mathcal{E}_{2}(\Gamma_{0}(N))$, it is necessary for
$N$ to be composite and for $\Gamma_{0}(N)$ to have no elliptic points.
Our next result shows that for most small $N$, this is sufficient.

\begin{thm}
\label{weakgen}
If $N \leq 300$ is composite, $\Gamma_{0}(N)$ has no elliptic
points, and $N \not\in \{121,209\}$,
then $\mathcal{R}_{2}(\Gamma_{0}(N)) = \mathcal{E}_{2}(\Gamma_{0}(N))$.
\end{thm}

\begin{rem}
It appears likely that for $N = 121$ and $N = 209$, the equivalent conditions
of Theorem~\ref{weakthm} are false.
\end{rem}

There are two significant consequences of being able to span spaces of
modular forms using eta-quotients.  First, this provides a means of
computing the Fourier expansions of a basis for
$M_{k}(\Gamma_{0}(N))$. The theory of modular symbols can also be used
to do this, but this process is somewhat inefficient in that it is
first necessary to compute the Fourier expansions of Hecke eigenforms,
and then to build a basis for $M_{k}(\Gamma_{0}(N))$ using
those. There are a number of situations where we are able to quickly
construct bases for spaces $M_{k}(\Gamma_{0}(N))$ with dimension
greater than $2000$ where the modular symbols algorithm is much less
efficient. Second, if $f(z) \in M_{k}(\Gamma_{0}(N))$, it is sometimes
desirable to compute the Fourier expansion of $f(z)$ at other cusps of
$\Gamma_{0}(N)$. There are no general algorithms to accomplish this
task. However, if $f(z)$ is an eta-quotient (or a linear combination
of eta-quotients), it is straightforward to do this, since the
transformation formula for $\eta(z)$ is known for every matrix in
$\SL_{2}(\Z)$.

\section{Modular forms non-zero on $\H$}

Looking at \eqref{eta-def}, it is not difficult to see that any eta-quotient must have integral Fourier coefficients and be non-zero on the upper-half plane.   We will examine to what degree the converse is true.  Kohnen \cite{kohnen_class_mod_funcs} studied a related problem, which we now describe.  Let
\[ f(z) = c q^h \prod_{n\geq 1} (1-q^n)^{c(n)} \in M_k^!(\Gamma_0(N)) \]
where $M_k^!(\Gamma_0(N))$ denotes the space of weight $k$ weakly holomorphic modular forms on $\Gamma_0(N)$, $c$ is a non-zero constant, $h \in \Z$ and each $c(n) \in \C$.  Then he proves the following.
\begin{thmnonum}[Theorem 2 in \cite{kohnen_class_mod_funcs}]
Suppose that $N$ is square-free.  Then $f$ has no zeros or poles on $\H$ if and only if $c(n)$ depends only on the greatest common divisor $(n,N)$.
\end{thmnonum}
A direct consequence of his proof is that for some $f(z) \in M_k^!(\Gamma_0(N))$ that has no zeros or poles on $\H$ and where $N$-square free, there exists some positive integer $t$ such that $f(z)^t$ is a constant times an eta-quotient.  If we replace the condition that $N$ is square-free with $f(z)\in \Z[[q]]$, we prove the following result.

\begin{thm}  \label{etaquotientthm} Suppose $f(z) \in M_k(\Gamma_0(N))\cap \Z[[q]]$ has the property that $f$ is non-zero on $\H$.  Then $f(z)=c g(z)$ where $c\in \Z$ and $g(z)$ is an eta-quotient.
\end{thm}

Let $\kappa$ be an even integer and $g(z) \in M_\kappa^!(\Gamma_0(N))\cap \Z((q))$ be non-zero on $\H$.  Since $\Delta(z)=\eta(z)^{24}$ is zero at every cusp of $\Gamma_0(N)$, for sufficiently large $t$, $g(z)\Delta(z)^t$ is a holomorphic modular form on $\Gamma_0(N)$ with integral Fourier coefficients and non-zero on the upper-half plane.  By Theorem \ref{etaquotientthm}, this implies $g(z)\Delta(z)^t$ is a constant times an eta-quotient which then implies the following.

\begin{cor} \label{weaklyholoetaquotient}
Suppose $g(z) \in M_\kappa^!(\Gamma_0(N))\cap \Z((q))$ has the property that $g$ is non-zero on $\H$.  Then $g(z)=c h(z)$ where $c\in \Z$ and $h(z)$ is an eta-quotient.
\end{cor}

On the other hand, the condition $f(z) \in \Z[[q]]$ in Theorem
\ref{etaquotientthm} is necessary.  To see this, suppose that $p$ is
an odd prime with $p^2 \mid N$.  Then by the a theorem of Manin and
Drinfeld \cite{drinfeld}, \cite{manin}, the cuspidal subgroup of the
Jacobian of $X_0(N)$ is finite which implies that there exists a $g(z)
\in M_0^! (\Gamma_0(N))$ whose zeros are only supported at the cusp
$1/p$ and whose poles are only supported at the cusp $\infty$.  It
follows that for a sufficiently large $t$, we have $f(z) = g(z)
\Delta(z)^t \in M_{12t}(\Gamma_0(N))$.  Such a form is non-zero on
$\H$.  Since $\Delta(z)=\eta(z)^{24}$, the remark following the
theorem of Ligozat implies that $\ord_{\frac1{p}} \Delta(z) =
\ord_{\frac{2}{p}} \Delta(z)$.  We also have $\ord_{\frac1{p}} g(z)
>0$ and $\ord_{\frac2{p}} g(z) = 0$.  It follows that $
\ord_{\frac1{p}}(f(z)) > \ord_{\frac2{p}}(f(z))$.  Using the remark
after the theorem of Ligozat again, we conclude that $f(z)$ is not an
eta-quotient.  By Corollary \ref{weaklyholoetaquotient}, such a
modular function $g(z)$ must have non-integral Fourier coefficients. Another way
of seeing this it to note that the cusps $\frac{1}{p}$ and $\frac{2}{p}$
on $X_{0}(N)$ are not defined over $\Q$ and are in the same orbit
under $\Gal(\overline{\Q}/\Q)$. Therefore, a modular form with
a zero at one and no zero at the other cannot be fixed by $\Gal(\overline{\Q}/\Q)$ and this implies it cannot have rational Fourier coefficients.

The multiplicative group of $M_0^!(\Gamma_0(N))$ is often called the
modular units on $\Gamma_0(N)$.  If $f(z)$ is a modular unit, then by definition $f(z)^{-1} \in M_0^{!}(\Gamma_0(N))$ as well, which implies that $f$ must be non-zero on $\H$.  Consider $\mathcal{M}(N) =
M_0^!(\Gamma_0(N))^{\times}\cap \Z((q))^{\times}$.  This is the
subgroup of the modular units on $\Gamma_0(N)$ with integral
coefficients and a leading coefficient of $\pm 1$ .  A special case of Corollary \ref{weaklyholoetaquotient} is the following.

\begin{cor}\label{modularunitscor}
Suppose $g(z) \in \mathcal{M}(N)$.  Then $g(z) = \pm f(z)$ where $f(z)$ is an eta-quotient.
\end{cor}

Our next result concerns the structure of the rational cuspidal subgroup $J_0(N)(\Q)_{\text{cusp}}$ when $N$ is a power of 2.  Let $\DIV^0 (X_0(N)/\Q)$ denote the degree 0 divisors of the modular curve $X_0(N)$ fixed by every element of $\Gal(\overline{\Q}/\Q)$.  Then $J_0(N)(\Q)_{\text{cusp}}$ is isomorphic to $\DIV^0 (X_0(N)/\Q) / \Div(\mathcal{M}(N))$.  Ling \cite{Ling} computed the structure of $J_0(p^k)(\Q)_{\text{cusp}}$ where $p\geq 3$ is prime and $k\geq 1$.  Using Corollary~\ref{modularunitscor}, we compute $J_0(2^k)(\Q)_{\text{cusp}}$.  Let $k\geq 1$ and define the set $I_k$ as follows.  For $1\leq k \leq 6$, define
\[ I_k = \begin{cases} \emptyset & \text{ if } 1\leq k \leq 4 \, ,\\
\{2\} &  \text{ if } k=5\, ,\\
\{1,2,2\} & \text{ if } k=6  \, .\end{cases}
\]

For $k\geq 7$, if $k$ is odd, define
\[i_{t,k} = \begin{cases} \lfloor \tfrac{t-1}{2} \rfloor + \tfrac{k-1}2 - 2 & \text{ if } 1 \leq t \leq k-1 \, , \\
k-3  & \text{ if } t=k-2\, , \end{cases} \]
and if $k$ is even, define
\[ i_{t,k} = \begin{cases} \lfloor \tfrac{t}2 \rfloor + \tfrac{k}2 - 3 & \text{ if } 1\leq t \leq k-4 \, ,\\
 k-4 & \text{ if } t = k-3 \text{ or } k-2 \, . \end{cases} \]
Then let
\[ I_k = \{ i_{t,k}\}_{t=1}^{k-2} \, .\]

\begin{thm}\label{level-2^k-rational-cusp-sbgrp-structure}
For $k \geq 1$,
\[ J_0(2^k)(\Q)_{\text{cusp}} \cong \bigoplus_{i\in I_k} \Z/2^i\Z \, .\]
\end{thm}

We now give a formula for the number of weight $k$ eta-quotients
of level $N$ assuming that $X_{0}(N)$ has genus zero (and hence
$J_{0}(N)(\Q)_{\text{cusp}}$ is trivial).

\begin{thm}
\label{genuszero}
Assume that $N$ is a positive integer so that $X_{0}(N)$ has genus zero,
and there is a holomorphic eta-quotient in $M_{k}(\Gamma_{0}(N))$. Then, the
number of eta-quotients in $M_{k}(\Gamma_{0}(N))$ is equal to the number of
tuples of non-negative integers $(c_{d} : d | N)$ with
\[
  \sum_{d | N} c_{d} \phi\left(\gcd\left(d,\frac{N}{d}\right)\right) = \frac{k}{12}
  \cdot [\SL_{2}(\Z) : \Gamma_{0}(N)].
\]
Here $\phi(n)$ is the usual Euler totient function.
\end{thm}

\begin{rem} As a consequence, there are $\frac{k^{2}}{8} + \frac{3k}{4} + 1$
eta-quotients in $M_{k}(\Gamma_{0}(4))$, $\frac{k^{3}}{6} + k^{2} + \frac{11k}{6} + 1$ eta-quotients in $M_{k}(\Gamma_{0}(8))$, and $\frac{k^{4}}{3} + 2k^{3} + \frac{25k^{2]}}{6} + \frac{7k}{2} + 1$ eta-quotients in $M_{k}(\Gamma_{0}(16))$.
\end{rem}

{\bf Acknowledgments} 
We used Magma~\cite{Magma} version 2.19-8 and LattE~\cite{LattE}
version 1.5 for computations. Details about the computations that were done
are available at\\ {\tt http://users.wfu.edu/rouseja/eta/}. The authors
wish to thank the anonymous referees for helpful comments which have
improved the exposition.

\section{Background}

Given a prime number $p$ and a non-zero integer $n$, we define
$\ord_{p}(n)$ to be the highest power of $p$ that evenly divides $n$.

For a positive even integer $k$, let
\[
  E_{k}(z) = 1 - \frac{2k}{B_{k}} \sum_{n=1}^{\infty} \sigma_{k-1}(n) q^{n}
\]
be the usual weight $k$ Eisenstein series, where $\sigma_{k-1}(n)
= \sum_{d | n} d^{k-1}$ is the sum of the $k-1$st powers of the divisors of $n$. Here $B_{k}$ is the $k$th Bernoulli
number. If $k \geq 4$, then $E_{k}(z) \in M_{k}(\Gamma_{0}(1))$. Let
\[
  j(z) = \frac{E_{4}(z)^{3}}{\Delta(z)} = q^{-1} + 744q + 196884q + \cdots
\]
be the usual modular $j$-function. If $d$ is a positive integer,
define the operator $V(d)$ by
\[
  f(z) = \left(\sum_{n=0}^{\infty} a(n) q^{n}\right) | V(d)
  = f(dz) = \sum_{n=0}^{\infty} a(n) q^{dn}.
\]
It is well-known that if $f(z) \in M_{k}(\Gamma_{0}(N))$,
then $f(z) | V(d) \in M_{k}(\Gamma_{0}(dN))$. (See for example Proposition~2.22
of \cite{Ono}.)

A level $N$ modular function is a meromorphic modular form of weight zero
for $\Gamma_{0}(N)$. It is known that a level $N$ modular function is a rational
function in $j(z)$ and $j(Nz)$. (See Proposition 7.5.1 of Diamond and
Shurman's book \cite{DS}, page 279.)

An elliptic point of $\Gamma_{0}(N)$ is a number $z \in \mathbb{H}$ so
that there is an $M \in \Gamma_{0}(N)$ with $Mz = z$ and $M \ne \pm
I$. In this case $M$ has order $2$ or $3$ in ${\rm PSL}_{2}(\Z)$, and
we call $z$ an elliptic point of order $2$ or $3$. Let
$\epsilon_{2}(\Gamma_{0}(N))$ denote the number of $\Gamma_{0}(N)$
orbits of elliptic points of order $2$, and
$\epsilon_{3}(\Gamma_{0}(N))$ the number of orbits of elliptic points
of order $3$. Diamond and Shurman (\cite{DS}, Corollary 3.7.2, page
96) give the formulas
\begin{align*}
  \epsilon_{2}(\Gamma_{0}(N)) &= \begin{cases}
  \prod_{p | N} \left(1 + \legen{-1}{p}\right) & \text{ if } 4 \nmid N\\
  0 & \text{ if } 4 | N \end{cases}\\
  \epsilon_{3}(\Gamma_{0}(N)) &= \begin{cases}
  \prod_{p | N} \left(1 + \legen{-3}{p}\right) & \text{ if } p \nmid N\\
  0 & \text{ if } 9 | N. \end{cases}
\end{align*}

The action of $\Gamma_{0}(N)$ on the upper half-plane extends to an action
on $\mathbb{P}^{1}(\Q)$. The cusps of $\Gamma_{0}(N)$ are the $\Gamma_{0}(N)$-orbits
of $\mathbb{P}^{1}(\Q)$. Given a number $(a : c) \in \mathbb{P}^{1}(\Q)$ (with $a, c \in \Z$
and $\gcd(a,c) = 1$), there is a matrix $M = \left[ \begin{matrix} a & b \\ c & d \end{matrix} \right] \in \SL_{2}(\Z)$ with $M(\infty) = a/c$. If
$f \in M_{k}(\Gamma_{0}(N))$, the order of vanishing of $f$ at the cusp $a/c$ is
the smallest power of $q$ with a nonzero coefficient in the expansion of
\[
  (cz+d)^{-k} f\left(\frac{az+b}{cz+d}\right) =
  \sum_{n=0}^{\infty} a(n) q^{n/N}.
\]
This will be denoted by $\ord_{a/c}(f)$. It does not depend on the matrix $M$ chosen, and if $(a : c)$ and $(a' : c')$ are $\Gamma_{0}(N)$-equivalent, then
$\ord_{a/c}(f) = \ord_{a'/c'}(f)$.

The graded ring of modular forms for $\Gamma_{0}(N)$ is $\bigoplus_{\substack{k \geq 0 \\ k \text{ even}}} M_{k}(\Gamma_{0}(N))$. We require some results about
the weights in which this graded ring is generated. Let $\mathcal{L}$
be the sheaf of $1$-forms on $X_{0}(N)$. Since $X_{0}(N)$
is non-singular, $\mathcal{L}$ is invertible and
$M_{k}(\Gamma_{0}(N))$ is isomorphic to $H^{0}(X_{0}(N), \mathcal{L}^{\otimes k/2})$.
Multiplication of modular forms corresponds to maps
\[
  H^{0}(X_{0}(N), \mathcal{L}^{\otimes a}) \otimes H^{0}(X_{0}(N), \mathcal{L}^{\otimes b})
  \to H^{0}(X_{0}(N), \mathcal{L}^{\otimes a+b}).
\]
In \cite{Mumford}, page 55, Mumford shows that if $\mathcal{L}$
is an invertible sheaf on a curve of genus $g$ and with degree $\geq 2g + 1$,
then $L$ is ample with normal generation. This implies that the tensor
product map above is surjective, for all $a, b \geq 1$ and hence
the graded ring of modular forms is generated in weight $2$. This
result seems to have been rediscovered by Rustom \cite{NR} and Khuri-Makdisi 
(see Proposition 2.1 of \cite{KM}).

If $\Gamma_{0}(N)$ has no elliptic points, the degree of the invertible
sheaf $\mathcal{L}$ is $2g - 2 + \epsilon_{\infty}$ where $\epsilon_{\infty}$ is the
number of cusps of $\Gamma_{0}(N)$.  If $N$ is composite, then
$\epsilon_{\infty} \geq 3$. Hence, if $N$ is composite we have that
the multiplication map $M_{k}(\Gamma_{0}(N)) \times
M_{l}(\Gamma_{0}(N)) \to M_{k+l}(\Gamma_{0}(N))$ is surjective for all
$k$ and $l$ with $k, l \geq 2$. This shows that the graded ring of modular
forms is generated in weight $2$ as long as $N$ is composite and $\Gamma_{0}(N)$
has no elliptic points.

\section{Proofs}
%%%%%Proof of Theorem 1%%%%%
First, we will analyze when there are eta-quotients of weight $k$
for $\Gamma_{0}(N)$. This is determined by the presence of elliptic points for $\Gamma_{0}(N)$. Recall that $\mathcal{E}_{k}(\Gamma_{0}(N))$ is the space
of weakly-holomorphic eta quotients of weight $k$ and level $N$.

\begin{lem}
\label{whenetas}
Fix a positive integer $N$. Then for even positive integers $k$,
$\mathcal{E}_{k}(\Gamma_{0}(N))$ is non-empty precisely when
\begin{enumerate}
\item $k \equiv 0 \pmod{12}$ if $\epsilon_{2}(\Gamma_{0}(N)) > 0$ and
$\epsilon_{3}(\Gamma_{0}(N)) > 0$,
\item $k \equiv 0 \pmod{6}$ if $\epsilon_{2}(\Gamma_{0}(N)) = 0$ and
$\epsilon_{3}(\Gamma_{0}(N)) > 0$,
\item $k \equiv 0 \pmod{4}$ if $\epsilon_{2}(\Gamma_{0}(N)) > 0$
and $\epsilon_{3}(\Gamma_{0}(N)) = 0$, and
\item $k \equiv 0 \pmod{2}$ if $\epsilon_{2}(\Gamma_{0}(N)) =
\epsilon_{3}(\Gamma_{0}(N)) = 0$.
\end{enumerate}
Moreover, if $\mathcal{E}_{k}(\Gamma_{0}(N))$ is non-empty, then
there is a holomorphic eta-quotient in $M_{k}(\Gamma_{0}(N))$.
\end{lem}
\begin{proof}
%A straightforward argument using the transformation law, together with
%the fact that every elliptic point of order $2$ has the form
%$\frac{ai+b}{ci+d}$ and every elliptic point of order $3$ has the form
%$\frac{a\omega + b}{c\omega + d}$ for $\left[ \begin{matrix} a & b \\ c & d \en%d{matrix}\right] \in \SL_{2}(\Z)$, where $\omega = e^{2 \pi i /3}$ shows that
%a weakly holomorphic modular form of weight $k$ must
%vanish at an elliptic point of order $3$ if $k \not\equiv 0 \pmod{6}$,
%and such a form must vanish at an elliptic point of order $2$ if $k \not\equiv
%0 \pmod{4}$.
First, suppose that $f$ is a weakly holomorphic modular form
of weight $k$ for $\Gamma_{0}(N)$. If $z$ is an elliptic point for
$\Gamma_{0}(N)$, then $\frac{az+b}{cz+d} = z$ for $M = \left[ \begin{matrix}
a & b \\ c & d \end{matrix} \right]$. It follows from this that
$|cz+d| = 1$. The transformation law gives that
\[
  f(z) = (cz+d)^{k} f(z).
\]
Thus, either $(cz+d)^{k} = 1$ or $f(z) = 0$. Corollary 2.3.4 of \cite{DS}
(page 55) implies that any elliptic point of order $2$ has the form
$\frac{ai+b}{ci+d}$ or $\frac{a \omega + b}{c \omega + d}$ and
where $\omega = e^{2 \pi i / 3}$. It follows from this fact
that if $z$ is an elliptic point of order $3$, then
$cz+d = \omega$ or $1 - \omega$, and if $z$ is an elliptic point of order $2$
then $cz+d = i$. It follows that if $k \not\equiv 0 \pmod{6}$ and
$z$ is an elliptic point of order $3$, then $f(z) = 0$. Also,
if $k \not\equiv 0 \pmod{4}$ and $z$ is an elliptic point of order $2$,
then $f(z) = 0$.

Assuming that $f(z)$ is a weakly holomorphic eta-quotient,
then we must have that $f(z)$ is non-vanishing on $\mathbb{H}$ and this
shows that in order for an eta-quotient to exist, the conditions in the
lemma must be satisfied.

It suffices to construct holomorphic eta-quotients in the remaining cases.
Since $\Delta(z) \in M_{12}(\Gamma_{0}(N))$ for all $N$, this shows that
there are always eta-quotients of weight $12$.

Assume that $\epsilon_{2}(\Gamma_{0}(N)) = 0$. Then either $4 | N$,
in which case $\frac{\eta(z)^{8}}{\eta(2z)^{4}} \in M_{2}(\Gamma_{0}(N))$,
or there is a prime $p | N$ with $p \equiv 3 \pmod{4}$. In this
case $\eta(z)^{6} \eta(pz)^{6} \in M_{6}(\Gamma_{0}(N))$ and so there are
eta-quotients of weight $6$.

Assume that $\epsilon_{3}(\Gamma_{0}(N)) = 0$. Then either $9 | N$,
in which case $\frac{\eta(z)^{6}}{\eta(3z)^{2}} \in M_{2}(\Gamma_{0}(N))$,
or there is a prime $p | N$ with $p \equiv 2 \pmod{3}$. If $p > 2$
then $p \equiv 5 \pmod{6}$ and $\eta(z)^{4} \eta(pz)^{4} \in M_{4}(\Gamma_{0}(N))$. If $p = 2$ then $\frac{\eta(z)^{16}}{\eta(2z)^{8}} \in M_{4}(\Gamma_{0}(N))$.
Thus, there is an eta-quotient of weight $4$ when $\epsilon_{3}(\Gamma_{0}(N)) = 0$.

Assume now that $\epsilon_{2}(\Gamma_{0}(N)) = \epsilon_{3}(\Gamma_{0}(N)) = 0$.
In the cases that $4 | N$ or $9 | N$ we have already constructed weight
$2$ eta-quotients. Otherwise, there is a prime $p | N$ with $p \equiv 3 \pmod{4}$ and a prime $q | N$ with $q \equiv 2 \pmod{3}$. If $p = q$,
then $p \equiv 11 \pmod{12}$ and $\eta(z)^{2} \eta(pz)^{2} \in M_{2}(\Gamma_{0}(N))$. If $p \ne q$ and $q > 2$ then $\eta(z) \eta(pz) \eta(qz) \eta(pqz) \in M_{2}(\Gamma_{0}(N))$, and if $p \ne q$ and $q = 2$ then $\frac{\eta(z)^{4} \eta(pz)^{4}}{\eta(2z)^{2} \eta(2pz)^{2}} \in M_{2}(\Gamma_{0}(N))$. Thus, there is
an eta-quotient of weight $2$ when $\epsilon_{2}(\Gamma_{0}(N)) = \epsilon_{3}(\Gamma_{0}(N)) = 0$.
\end{proof}

Now, we will prove Theorem~\ref{hologen}.

\begin{proof}[Proof of Theorem~\ref{hologen}]
  If the graded ring of modular forms for $\Gamma_{0}(N)$ is generated
  by eta-quotients, then $M_{2}(\Gamma_{0}(N))$ must be spanned by
  eta-quotients.  Since $M_{2}(\Gamma_{0}(N))$ has positive dimension
  if $N > 1$, and there can be no weight $2$ eta-quotients if
  $\Gamma_{0}(N)$ has elliptic points by Lemma~\ref{whenetas}, it
  follows that $\Gamma_{0}(N)$ has no elliptic points. (The only level
  $1$ eta-quotients are powers of $\Delta(z)$, so the same conclusion
  follows when $N = 1$.)

In each case, we attempt to find enough eta-quotients to span
$M_{2}(\Gamma_{0}(N))$. If we succeed and $N$ is composite, then
the fact that the graded ring of level $N$ modular forms is generated
in weight $2$ proves the desired result. In cases where we fail
to find enough eta-quotients, we enumerate all vectors in the corresponding
$d(N) - 1$-dimensional lattice that correspond to weight $2$ eta-quotients,
and check that the dimension of the space they span is less than
$\dim M_{2}(\Gamma_{0}(N))$ to prove that $M_{2}(\Gamma_{0}(N))$ is not
generated by eta-quotients.

There are no prime levels $N \leq 500$ for which $M_{2}(\Gamma_{0}(N))$
is generated by eta-quotients. (In fact, there are no prime levels whatsoever,
by Theorem~\ref{finitecodim}.) Details about the computations are available
at {\tt http://users.wfu.edu/rouseja/eta/}.
\end{proof}
%%%%%%%%%%%%%%%%Proof of Theorem 2
Before we prove Theorem~\ref{expobound}, we need two lemmas first.
\begin{lem}
\label{vlem}
Suppose that $f(z) = \prod_{\delta | N} \eta(\delta z)^{r_{\delta}}$ is a level
$N$ eta quotient. Suppose that $e$ is a positive integer and
$\gcd(e,N) = 1$. Let $r | e$ and view $f(z) | V(r) = \prod_{\delta | N} \eta(r \delta z)^{r_{\delta}}$ as an eta quotient of level $eN$. Let $d$ be a divisor of $eN$ and write $d = d_{1} d_{2}$, where $d_{1} | N$ and $d_{2} | e$. Then we have
\[
  \ord_{\frac{1}{d}}(f(z) | V(r)) = \frac{e \gcd(d_{2}^{2}, r^{2})}{r \gcd(d_{2}^{2}, e)}
  \ord_{\frac{1}{d_{1}}}(f(z)).
\]
\end{lem}
\begin{proof}
We have $f(z) | V(r) = \prod_{\delta | eN} \eta(\delta z)^{s_{\delta}}$,
where $s_{\delta} = r_{\delta/r}$ if $r | \delta$ and $\delta/r | N$. Otherwise
$s_{\delta} = 0$. The order of vanishing at the cusp $1/d$ is then
\begin{align*}
  & \frac{Ne}{24} \sum_{\substack{\delta | Ne \\ r | \delta}}
  \frac{\gcd(d,\delta)^{2} s_{\delta}}{\gcd\left(d,\frac{Ne}{d}\right) d \delta}
  = \frac{Ne}{24} \sum_{\delta | N}
  \frac{\gcd(d,r \delta)^{2} r_{\delta}}{\gcd\left(d,\frac{Ne}{d}\right) dr \delta}\\
  &= \frac{Ne}{24r} \sum_{\delta | N}
  \frac{\gcd(d_{1}, \delta)^{2} \gcd(d_{2},r)^{2} r_{\delta}}{\gcd\left(d_{1},\frac{N}{d_{1}}\right) \gcd\left(d_{2},\frac{e}{d_{2}}\right) d_{1} d_{2} \delta}
  = \frac{e \gcd(d_{2}^{2}, r^{2})}{r \gcd(d_{2}^{2}, e)} \cdot
  \left(\frac{N}{24} \sum_{\delta | N}
  \frac{\gcd(d_{1}, \delta)^{2} r_{\delta}}{\gcd\left(d_{1},\frac{N}{d_{1}}\right) d_{1} \delta}\right)\\
  &= \frac{e \gcd(d_{2}^{2}, r^{2})}{r \gcd(d_{2}^{2}, e)} \ord_{\frac{1}{d_{1}}}(f(z)).
\end{align*}
\end{proof}

The eta-quotients constructed in the next lemma will play a role in the proofs
of many of the theorems.

\begin{lem}
\label{magicetas}
If $N$ is a positive integer, then for each divisor $d$ of $N$, there is
a holomorphic eta-quotient $E_{d,N}(z) \in M_{k_{d}}(\Gamma_{0}(N))$ that vanishes
only at the cusps $c/d$. Moreover, if $E_{d,N}(z) = \prod_{\delta | N} \eta(\delta z)^{r_{\delta}}$, we have
\[
  \frac{1}{2k_d} \sum_{\delta | N} |r_{\delta}| \leq \prod_{p | N} \left(\frac{p+1}{p-1}\right)^{\min \{2,\ord_{p}(N)\}}.
\]
We follow the convention that an empty product on the right hand side
takes the value $1$.
\end{lem}
\begin{proof}
We will prove our result by induction on the number of distinct prime factors
of $N$. The base case is $N = 1$. In this case we have $E_{1,1}(z) =
\Delta(z)$ which vanishes at the only cusp of $\Gamma_{0}(1)$ and we have
equality in the claimed inequality.

Assume now that $N$ is a positive integer with $k$ distinct prime factors,
one of which is $p$ and the highest power of $p$ dividing $N$ is $p^{m}$.
By induction, for each divisor $d$ of $N/p^{m}$, there is a form
\[
  F(z) = E_{d,N/p^{m}}(z)
\]
that has all of its zeros at cusps with denominator $d$ on
$\Gamma_{0}(N/p^{m})$.

Let $E_{d,N}(z) = \frac{F(z)^{p}}{F(z) | V(p)}$ and let $d_1$ and $d_2$ be divisors of $N/p^m$ and $p^m$ respectively. We now apply Lemma~\ref{vlem}
with $e = p^{m}$, $r = 1$ or $p$.  Then we get that
\[
  \ord_{1/d_{1} d_{2}}(E_{d,N}(z))
  = \left(\left(\frac{p^{m+1}}{\gcd(d_{2}^{2},p^{m})}\right) - \left(\frac{p^{m} \gcd(d_{2}^{2}, p^{2})}{p \gcd(d_{2}^{2},p^{m})}\right)\right) \ord_{1/d_{1}}(E_{d,N}(z)).
\]
It is clear that this quantity is always non-negative. Moreover, the
expression in parentheses is equal to zero if $p | d_{2}$. If $d_{1}
\ne d$, then $\ord_{1/d_{1}}(E_{d,N}(z)) = 0$. This implies that all
of the zeros of $E_{d,N}(z)$ occur at the cusps with denominator
$d$, as desired. The passage from $F(z)$ to $E_{d,N}(z)$ multiplies
the weight by $p-1$ and multiplies the sum of the absolute values of
the exponents by $p+1$.

For $1 \leq s \leq m-1$, define
\[
  E_{dp^{s},N}(z) = \frac{(F(z) | V(p^s))^{p^{2} + 1}}{(F(z) | V(p^{s-1}))^{p} (F(z) | V(p^{s+1}))^{p}}.
\]
Recall that $d_{1}$ is a divisor of $N/p^{m}$, $d_{2}$ is a divisor of $p^{m}$.
We apply Lemma~\ref{vlem} with $e = p^{m}$, $r = p^{s-1}$, $p^{s}$ or $p^{s+1}$
and we get that
\begin{align*}
  & \ord_{1/d_{1} d_{2}}(E_{dp^{s},N}(z))\\
  &= \left(\frac{(p^{2} + 1) \cdot p^{m} \gcd(d_{2}^{2},p^{2s})}{p^{s} \gcd(d_{2}^{2},p^{m})} - \frac{p \cdot p^{m} \gcd(d_{2}^{2},p^{2s-2})}{p^{s-1} \gcd(d_{2}^{2}, p^{m})} - \frac{p \cdot p^{m} \gcd(d_{2}^{2},p^{2s+2})}{p^{s+1} \gcd(d_{2}^{2}, p^{m})}\right) \ord_{1/d_{1}}(F(z))\\
  &= p^{m} \left(\frac{(p^{2} + 1) \gcd(d_{2}^{2},p^{2s}) - p^{2} \gcd(d_{2}^{2},p^{2s-2}) - \gcd(d_{2}^{2},p^{2s+2})}{p^{s} \gcd(d_{2}^{2}, p^{m})}\right) \ord_{1/d_{1}}(F(z)).\\
\end{align*}
If we write $d_{2} = p^{v}$, the numerator of the above fraction is
\[
  p^{\min\{2v+2,2s+2\}} + p^{\min\{2v,2s\}} - p^{\min\{2v+2,2s\}} -
  p^{\min\{2v,2s+2\}}.
\]
It is easy to see that if $s \ne v$, the above quantity is zero. If
$s = v$, it is $p^{2s+2} - p^{2s}$. This shows that $E_{dp^{s},N}(z)$ is
non-vanishing unless $d_{2} = p^{s}$ and $d_{1} = d$.  Moreover, the transfer
from $F(z)$ to $E_{dp^{s},N}(z)$ multiplies the weight by $(p-1)^{2}$
and the sum of the absolute values of the exponents by $(p+1)^{2}$. Hence,
the stated inequality is true.

Finally, we let $E_{dp^{m},N}(z) = \frac{(F(z) | V(p^{m}))^{p}}{F(z) | V(p^{m-1})}$.
We have
\[
  \ord_{1/d_{1} d_{2}}(E_{dp^{m},N}(z)) =
  \left(\frac{p \cdot p^{m} \gcd(d_{2}^{2}, p^{2m})}{p^{m} \gcd(d_{2}^{2}, p^{m})}
  - \frac{p^{m} \gcd(d_{2}^{2}, p^{2m-2})}{p^{m-1} \gcd(d_{2}^{2}, p^{m})}\right)
  \ord_{1/d_{1}}(F(z)).
\]
The quantity in parentheses is clearly zero unless $d_{2} = p^{m}$. If
$d_{2} = p^{m}$, it is $p^{m+1} - p^{m-1}$. This shows that
$E_{dp^{m},N}(z)$ is non-vanishing except at cusps with denominator
$dp^{m}$.  Moreover, the transfer from $F(z)$ to $E_{dp^{m},N}(z)$
multiplies the weight by $p-1$ and the sum of the absolute values of
the exponents by $p+1$. This proves the desired result by induction
on the number of distinct prime factors of $N$.
\end{proof}

\begin{proof}[Proof of Theorem~\ref{expobound}]
By Lemma~\ref{magicetas}, for each $d | N$, there is a form
$E_{d,N}(z)$ that vanishes only at the cusps $c/d$. We may therefore write an arbitrary eta-quotient $f(z)$ in the form
\[
  f(z) = \prod_{d | N} E_{d,N}(z)^{r_{d}}
\]
for a sequence of non-negative rational numbers $r_{d}$. Since the desired
inequality is true for each $E_{d,N}$, it must be valid for $f(z)$ as well.
\end{proof}

%%%%%%%%%%%Proof of Corollary 3
Next, we prove Corollary~\ref{primelevel}

\begin{proof}[Proof of Corollary~\ref{primelevel}]
Theorem~\ref{expobound} shows that if $p > 73$ and $f(z) = \prod_{\delta | 4p} \eta(\delta z)^{r_{\delta}} \in M_{2}(\Gamma_{0}(4p))$, then
$\sum_{\delta | 4p} |r_{\delta}| \leq 36$ and $\sum_{\delta | 4p} r_{\delta} = 4$. We
explicitly enumerate all such exponents and we find that
\[
  \text{ the number of eta-quotients } \in M_{2}(\Gamma_{0}(4p)) =
  \begin{cases}
     10 & \text{ if } p = 2\\
     126 & \text{ if } p = 3\\
     76 & \text{ if } p = 5\\
     45 & \text{ if } p = 7\\
     28 & \text{ if } p = 11\\
     9 & \text{ if } p \equiv 1 \pmod{24}\\
     16 & \text{ if } p \equiv 5, 11, \text{ or } 17 \pmod{24} \text{ and } p > 11\\
     21 & \text{ if } p \equiv 7 \pmod{24} \text{ and } p > 7\\
     15 & \text{ if } p \equiv 13 \pmod{24}\\
     18 & \text{ if } p \equiv 19 \pmod{24}\\
     37 & \text{ if } p \equiv 23 \pmod{24}.
\end{cases}
\]
A straightforward calculation with this data and with the dimension
formula for $M_{2}(\Gamma_{0}(4p))$ shows that if $p > 47$ then $\dim M_{2}(\Gamma_{0}(4p))$ is larger than the number of eta quotients it contains. Explicit
calculations for $p \leq 73$ prove the desired result.
\end{proof}

%%%%%%%%%%%%%%%Proof of Theorem 4
\begin{proof}[Proof of Theorem~\ref{finitecodim}]
We will first prove that $\mathcal{E}_{0}(\Gamma_{0}(N))$ has
finite codimension in $\mathcal{R}_{0}(\Gamma_{0}(N))$ when $N$ is composite.
Then $N$ has at least three divisors, say $1$, $d$ and $e$ with $1 < d < e$.

Let $E_{N,N}(z)$ be the eta-quotient from Lemma~\ref{magicetas} that has zeros only at $\infty$, and let $r$ be the weight of $E_{N,N}(z)$. Define
\[
  f_{1}(z) = \frac{\Delta(z)^{re-e+1} \Delta(dz)^{e-1}}{E_{N,N}(z)^{12e}} \text{ and }
  f_{2}(z) = \frac{\Delta(z)^{re-d+1} \Delta(ez)^{d-1}}{E_{N,N}(z)^{12e}}.
\]
These are both modular functions for $\Gamma_{0}(N)$ and the orders of the poles
of $f_{1}(z)$ and $f_{2}(z)$ at infinity are both
\[
  M =  12e \ord_{\infty}(E_{N,N})-(ke + de - d - e + 1) .
\]
Since $\Delta(z) = q - 24q^{2} + O(q^{3})$, it follows
that $\Delta(z)^{s} = q^{s} - 24sq^{s+1} + O(q^{s+2})$, and since the
power of $\Delta(z)$ in $f_{1}(z)$ and $f_{2}(z)$ are different, the coefficients
of $q^{-M+1}$ in $f_{1}(z)$ and $f_{2}(z)$ are different.

It follows that $\mathcal{E}_{0}(\Gamma_{0}(N))$ contains $f_{1}(z)$,
which has a pole of order $M$ at infinity, and $f_{1}(z) - f_{2}(z)$, which has
a pole of order $M-1$ at infinity.  Since every positive integer $n \geq M^{2}
- 3M+1$ can be written as a non-negative linear combination of $M$ and
$M-1$, the ring $\mathcal{E}_{0}(\Gamma_{0}(N))$ contains a function
with pole of order $n$ at infinity for every $n \geq M^{2}- 3M+1$. It follows
that every element of $\mathcal{R}_{0}(\Gamma_{0}(N))$ can be written in the form
\[
  f(z) = s(z) + t(z)
\]
where $s(z) \in \mathcal{E}_{0}(\Gamma_{0}(N))$ and $\ord_{\infty}
t(z) \geq -(M^{2} - 3M+1)$. The set of functions $t(z) \in
\mathcal{R}_{0}(\Gamma_{0}(N))$ with $\ord_{\infty} t(z) \leq -(M^{2} +3M+1)$ is a finite dimensional vector space (of dimension at most $M^{2} -3M + 1$) and this 
proves the result. Since
$\mathcal{E}_{k}(\Gamma_{0}(N))$ is a module for $\mathcal{E}_{0}(\Gamma_{0}(N))$,
we have that $\mathcal{E}_{k}(\Gamma_{0}(N))$ has finite codimension in
$\mathcal{R}_{k}(\Gamma_{0}(N))$ if $\mathcal{E}_{k}(\Gamma_{0}(N))$ is nonempty.
This occurs precisely when the conditions of Lemma~\ref{whenetas} are satisfied.

Suppose now that $N = p$ is prime. Any element in
$\mathcal{E}_{k}(\Gamma_{0}(p))$ is a linear combination of those
of the form $\eta(z)^{a} \eta(pz)^{2k-a}$ for some even integer $a$.
The order of vanishing of this form at infinity is $\frac{2kp - a(p-1)}{24}$.
We see that if $a$ and $a'$ are both such that
$f_{1}(z) = \eta(z)^{a} \eta(pz)^{2k-a}$ and $f_{2}(z) = \eta(z)^{a'} \eta(pz)^{2k-a'}$ are level $p$
eta-quotients, then the orders of vanishing of $f_{1}(z)$ and $f_{2}(z)$
at infinity must be congruent modulo $r = \frac{p-1}{\gcd(24,p-1)}$.

However, the Riemann-Roch theorem implies that if $g$ is the genus of
$X_{0}(N)$ and $m \geq 2g$, then there is an element in
$\mathcal{R}_{0}(\Gamma_{0}(p))$ with a pole of order $m$ at
infinity. The fact that $\mathcal{R}_{k}(\Gamma_{0}(p))$ is non-empty
and is a module for $\mathcal{R}_{0}(\Gamma_{0}(p))$ implies that
every sufficiently large positive integer occurs as the order of pole of
an element of $\mathcal{R}_{k}(\Gamma_{0}(p))$. This easily implies that if
\[
  r = \frac{p-1}{\gcd(24,p-1)} > 1
\]
then $\mathcal{E}_{k}(\Gamma_{0}(N))$ does not have finite codimension in $\mathcal{R}_{k}(\Gamma_{0}(N))$. We have
that $\frac{p-1}{\gcd(24,p-1)} = 1$ if and only if $p-1 | 24$ and this occurs
precisely for $p = 2, 3, 5, 7$ and $13$. If $p = 2, 3, 5, 7$ or $13$,
the fact that the order of vanishing of $f(z)$ at $\infty$ is $-1$ implies that $\mathcal{E}_{0}(\Gamma_{0}(p)) = \mathcal{R}_{0}(\Gamma_{0}(p))$.
This proves the desired result.
\end{proof}

%%%%%%%%%%% Proof of Theorem 5
\begin{proof}[Proof of Theorem~\ref{weakthm}]
  Again, let $E_{N,N}(z)$ be the holomorphic eta-quotient with zeros
  only at the cusp at $\infty$ from Lemma~\ref{magicetas}. Let the weight of $E_{N,N}$ be $r$.

  First, we will show that $(1) \implies (2)$. Suppose that for all
  positive integers $k$, every element of $M_{k}(\Gamma_{0}(N))$ is a
  linear combination of weakly holomorphic eta-quotients with poles
  only at infinity. Then there must be weight $2$ weakly holomorphic
  eta-quotients which implies (by Lemma~\ref{whenetas}) that $\Gamma_{0}(N)$
has no elliptic points. Suppose that $f \in \mathcal{R}_{k}(\Gamma_{0}(N))$. There is a
  positive integer $m$ for which $f E_{N,N}^{m}$ is holomorphic at
  infinity, and since $f$ only has poles at infinity, $f E_{N,N}^{m}
  \in M_{k+mr}(\Gamma_{0}(N))$.  By assumption, therefore, we may
  write
\[
  f E_{N,N}^{m} = \sum_{i} c_{i} g_{i}
\]
where the $g_{i}$ are weakly holomorphic eta-quotients with poles only
at infinity and so
\[
  f = \sum_{i} c_{i} \frac{g_{i}}{E_{N,N}^{m}} \in \mathcal{E}_{k}(\Gamma_{0}(N))
\]
is a linear combination of weakly holomorphic eta-quotients with poles
only at infinity, and so $\mathcal{R}_{k}(\Gamma_{0}(N)) = \mathcal{E}_{k}(\Gamma_{0}(N))$ for all positive even $k$. Observe that the map $T : \mathcal{R}_{k}(\Gamma_{0}(N)) \to \mathcal{R}_{r+k}(\Gamma_{0}(N))$ given by $T(f) = f E_{N,N}$ is an isomorphism. Similarly,
the map $T : \mathcal{E}_{k}(\Gamma_{0}(N)) \to \mathcal{E}_{r+k}(\Gamma_{0}(N))$
given by $T(f) = f E_{N,N}$ is also an isomorphism, and so
$\mathcal{R}_{k}(\Gamma_{0}(N)) = \mathcal{E}_{k}(\Gamma_{0}(N))$ if and only if
$\mathcal{R}_{k+r}(\Gamma_{0}(N)) = \mathcal{E}_{k+r}(\Gamma_{0}(N))$.
Hence, since $\mathcal{R}_{r}(\Gamma_{0}(N)) = \mathcal{E}_{r}(\Gamma_{0}(N))$,
it follows that $\mathcal{R}_{0}(\Gamma_{0}(N)) = \mathcal{E}_{0}(\Gamma_{0}(N))$.
Thus, $(1) \implies (2)$.

Note that $(2) \implies (1)$ is trivial, since
if $f(z) \in M_{k}(\Gamma_{0}(N))$, then
$f(z) \in \mathcal{R}_{k}(\Gamma_{0}(N)) = \mathcal{E}_{k}(\Gamma_{0}(N))$,
and so $f(z)$ is a linear combination of eta-quotients with poles only
at $\infty$.

It is obvious that $(2) \implies (3)$. We will now prove that
$(3) \implies (1)$.

For all $N \geq 1$, $\mathcal{R}_{2}(\Gamma_{0}(N))$ is non-zero,
since if $N > 1$ it contains $M_{2}(\Gamma_{0}(N))$, which has positive dimension, and for $N = 1$, it contains $\frac{E_{4}(z)^{2} E_{6}(z)}{\Delta(z)}$.
Therefore, the hypothesis that $\mathcal{R}_{2}(\Gamma_{0}(N)) = \mathcal{E}_{2}(\Gamma_{0}(N))$ implies that $\mathcal{E}_{2}(\Gamma_{0}(N))$ is non-empty. Lemma~\ref{whenetas} implies that $\Gamma_{0}(N)$ has no elliptic points. In particular, $N$ cannot equal $2$, $3$, $5$, $7$ or $13$.

Also, the hypothesis that $\mathcal{E}_{2}(\Gamma_{0}(N)) = \mathcal{R}_{2}(\Gamma_{0}(N))$ implies (by Theorem~\ref{finitecodim}) that $N \ne 11$ and $N$ cannot
be a prime greater than $13$. Thus, $N$ is composite.

Since $N$ is composite and $\Gamma_{0}(N)$ has no elliptic points,
the graded ring of modular forms for $\Gamma_{0}(N)$ is generated in weight $2$.
Since $M_{2}(\Gamma_{0}(N)) \subseteq \mathcal{R}_{2}(\Gamma_{0}(N)) =
\mathcal{E}_{2}(\Gamma_{0}(N))$, every weight $2$ modular form is a
linear combination of weakly holomorphic eta-quotients with poles only
at infinity. Thus, the graded ring of modular forms is generated by
forms that are linear combinations of weakly holomorphic eta-quotients,
and this proves that $(3) \implies (1)$.
\end{proof}

%%%%%%%%%Proof of Theorem 6
\begin{proof}[Proof of Theorem~\ref{weakgen}]
  Again, let $E_{N,N}$ be the level $N$ eta-quotient of weight $r$
  with zeros only at infinity from Lemma~\ref{magicetas}. In each
  case, we are able to prove that every form in
  $M_{r+2}(\Gamma_{0}(N))$ is a linear combination of (holomorphic) eta-quotients
and this proves that every element of $M_{2}(\Gamma_{0}(N))$ is a linear
combination of weakly holomorphic eta-quotients. Since the graded ring
of modular forms of level $N$ is generated in weight $2$, this proves
that every element of $M_{k}(\Gamma_{0}(N))$ is a linear combination of
weakly holomorphic eta-quotients. Then, by Theorem~\ref{weakthm},
we have that $\mathcal{R}_{2}(\Gamma_{0}(N)) = \mathcal{E}_{2}(\Gamma_{0}(N))$.
For code to compute eta-quotients in $M_{r+2}(\Gamma_{0}(N))$, see
{\tt http://users.wfu.edu/rouseja/eta/}.
\end{proof}

%%%%%%%%%%%%%%%%Proof of Theorem 7

Before we begin the proof of Theorem~\ref{etaquotientthm}, we will need a series of preliminary results.

\begin{lem}\label{step 1}
Let $g(z) \in M_k(\Gamma_0(N)) \cap \Q((q))$ have the property that $g$ is non-zero on the upper half plane and let $d| N$ with $d>0$.  Then for any $\ell\in \Z$ with $\gcd(\ell,d)=1$, we have $\ord_{\frac{1}{d}}(g(z)) = \ord_{\frac{\ell}{d}} (g(z))$.
\end{lem}

\begin{proof}Since $k\in 2\Z$, there exist integers $a,b$ such that $4a+6b = k$.  This implies that
\[ G(z) := \dfrac{g(z)}{E_4(z)^a E_6(z)^b} \]
is a meromorphic modular function for $X_0(N)$.  Since $E_4(z), E_6(z) \in \Z[[q]]$, $G(z) \in \Q((q))$ as well.  Thus for any $\tau \in \Gal(\overline{\Q}/\Q)$, we have $G(z)^\tau = G(z)$, and thus $\Div(G(z))^\tau = \Div(G(z))$.  Furthermore, at any cusp $\frac{\ell}{d} \in \Q$, we have $\ord_{\frac{\ell}{d}} (G(z)) = \ord_{\frac{\ell}{d}} (g(z))$.  We note that since $G(z)$ is a meromorphic function on $\Gamma_0(N)$ with rational coefficients, we have $G(z) \in \Q(j(z),j(Nz))$ where $j$ is the modular $j$-function.

We note that $\Q(j(z),j(Nz))$ is the function field for a model of the modular curve $X_0(N)/\Q$. For $d,\ell \in \Z$ with $\gcd(\ell,d)=1$, we let $\left[{\ell}\atop{d} \right]$ denote the point on $X_0(N)$ associated to the cusp $\frac{\ell}{d}$.  We will use the following result.

\begin{thm}[adapted from Theorem 1.3.1 in \cite{stevens}] \label{stevensthm}\,

\begin{enumerate}
\item \label{stevensa} The cusps of $X_0(N)$ are rational over $\Q(\zeta_N)$, where $\zeta_N = e^{2 \pi i/N}$.
\item  \label{stevensb} For $s\in(\Z/N\Z)^{\times}$ let $\tau_s \in \Gal(\Q(\zeta_N)/\Q)$ be defined by
\[ \tau_s : \zeta_N \mapsto \zeta_N^s \, .\]
Then
\[ \left[ {\ell}\atop{d} \right]^{\tau_s} = \left[ {\ell}\atop{s' d} \right]
\]
where $s'\in \Z$ is chosen so that $s s' \equiv 1 \pmod{N}$.
\end{enumerate}
\end{thm}

Since $\Div(G(z))$ is fixed by each element of $ \Gal(\Q(\zeta_N)/\Q)$, part \ref{stevensa} of Theorem \ref{stevensthm} implies that for any two cusps $\frac{j}{d}, \frac{j'}{d'}$ such that there exists some $\tau_s \in \Gal(\Q(\zeta_N)/\Q)$ with $\left[ {j}\atop{d} \right]^{\tau_s} = \left[ {j'}\atop{d'} \right]$, we have $\ord_{\frac{j}{d}}(G(z))= \ord_{\frac{j'}{d'}}(G(z))$.  We will finish the proof of the lemma by showing that if $d |N$ and $\ell \in \Z$ with $\gcd(d,\ell)=1$, then
\begin{equation}\label{steponelastpart}
\left[ 1\atop{d} \right]^{\tau_{\ell^*}} = \left[ {\ell}\atop{d}\right]
\end{equation}
where $\ell^* \in \Z$ with $\ell^*  \equiv \ell \pmod{d}$ and $\gcd(\ell^*,N)=1$.

Choose $\ell'$ such that $\ell^* \ell' \equiv 1 \pmod{N}$.  Then there exists some $t \in \Z$ such that $\ell^ *\ell' + tN =1$, and $A,B\in \Z$ such that $A \ell' - BdtN = 1$.  Consider the matrix
\[ \gamma := \left[\begin{matrix} A & B \\ dtN & \ell' \end{matrix}\right] \in \Gamma_0(N)\, .\]
One computes that
\[ \gamma\left(\frac{1}{\ell^*d}\right) = \frac{A+B\ell^*d}{d} \, .\]
Our choice of $A$ implies that $A\equiv \ell \pmod{d}$. Thus there exists some $r\in \Z$ such that
\[ \frac{A+B\ell^*d}{d} = \frac{\ell}{d} +r \, .\]
This implies that $\left[ {\ell}\atop{d} \right]=\left[ {1}\atop{\ell^* d}\right]$, which by part \ref{stevensb} of Theorem \ref{stevensthm} implies \eqref{steponelastpart}, finishing the proof of the lemma.
\end{proof}

\begin{lem}\label{step 2}
Suppose $f(z) \in M_k(\Gamma_0(N))\cap \Z[[q]]$ has the property that $f(z)$ is non-zero on the upper half plane.  Then there exists a positive integer $c$ such that $f(z)^c= a g(z)$ where $a \in \Z$ and $g(z)$ is an eta-quotient.
\end{lem}
\begin{proof}Let $d, j \in \Z$ with $d|N$ and $\gcd(j,n)=1$.  Then by Lemma \ref{step 1}, we have $\ord_{\frac1{d}}(f(z)) = \ord_{\frac{j}{d}} (f(z))$ for all such $d$ and $j$.  For each divisor $d$ of $N$, define
\[ c_d =\begin{cases} \dfrac{\lcm(\ord_{\frac{1}{d}}(f(z)),\ord_{\frac{1}{d}}(E_{d,N}(z)))}{\ord_{\frac{1}{d}}(f(z))} & \text{ if  } \ord_{\frac{1}{d}} (f(z))\neq 0\, , \\
1 & \text{ if } \ord_{\frac{1}{d}} (f(z)) = 0 \, , \end{cases}
\]
 and
\[ r_d =\begin{cases} \dfrac{\lcm(\ord_{\frac{1}{d}}(f(z)),\ord_{\frac{1}{d}}(E_{d,N}(z)))}{\ord_{\frac{1}{d}}(E_{d,N}(z))} & \text{ if } \ord_{\frac1{d}} (f(z)) \neq 0 \, , \\
0 &\text{ if } \ord_{\frac{1}{d}}(f(z)) = 0 \, ,\end{cases}
\]
where $E_{d,N}$ is defined as in Lemma~\ref{magicetas}.
Let $c= \prod c_d$.  Now define
\[ F(z) = \prod_{d \mid N} E_{d,N}(z)^{r_d c /c_d}\, .\]
One can check using Lemma \ref{step 1} and the definition of $E_{d,N}(z)$ that at each cusp $\frac{j}{d}$, we have
\[
\ord_{\frac{j}{d}}(f(z)^c) = \ord_{\frac{j}{d}}(F(z)) =\begin{cases}  \dfrac{c}{c_d} \lcm(\ord_{\frac{1}{d}}(f(z)),\ord_{\frac{1}{d}}(E_{d,N}(z))) & \text{ if } \ord_{\frac{1}{d}} (f(z))\neq 0\, , \\
0 & \text{ if } \ord_{\frac{1}{d}} (f(z)) = 0 \, . \end{cases}
\]
Since $f(z)^c$ and $F(z)$ are both holomorphic and non-zero on the upper-half plane, this implies that $f(z)^c / F(z) \in M_0^!(\Gamma_0(N)) \cap \Z((q))$ is a holomorphic non-zero function on the upper-half plane and at the cusps.  We conclude that there exists some $A \in \C$ such that $f(z)^c = A F(z)$, which completes the proof of the lemma.
\end{proof}

\begin{cor}\label{step 3}
Suppose $f(z)=\sum_{n=n_0}^\infty a(n)q^n \in M_k(\Gamma_0(N))\cap \Z[[q]]$ has the property that $f(z)$ is non-zero on the upper half plane.  Then $\frac{1}{a(n_0)} f(z) \in \Z[[q]]$, that is, every coefficient of $f(z)$ is divisible by the first non-zero coefficient.
\end{cor}

\begin{proof}Using the same notation as in the proof of Lemma \ref{step 2}, since the leading non-zero coefficient of the Fourier expansion of $F(z)$ is one, we have that $A=a(n_0)^c \in \Z$.  If $|a(n_0)|= 1$, the statement of the lemma is trivially true.  Now suppose $|a(n_0)|>1$, and let $p | a(n_0)$ be prime.  Let $f(z) = \sum_{n=n_0}^\infty a(n)q^n$ and define two series $S(q), R(q) \in \Z[[q]]$ by
\begin{align*} S(q) & = \sum_{n=n_0}^\infty \left\lfloor \frac{a(n)}{p} \right\rfloor q^n\, ,\\
R(q) & = \sum_{n=n_0}^\infty \left(a(n)- p\left\lfloor \frac{a(n)}{p} \right\rfloor \right) q^n \, .
\end{align*}
Note that each coefficient of $R(q)$ is in the set $[0, p-1]\cap \Z$.
Since $f(z)= pS(q)+R(q)$,
\[ \dfrac{1}{p} f(z)^c = \sum_{t=0}^{c} p^{t-1}\binom{c}{t}S(q)^tR(q)^{c-t} = \dfrac{a(n_0)^c}{p} F(z) \, .\]
For each $t\geq 1$, each summand above lies in $\Z[[q]]$.  Since $p|a(n_0)$ and $F(z) \in \Z[[q]]$, we see that $\frac{1}{p} f(z)^c \in \Z[[q]]$ as well.  This implies that the summand $\frac{1}{p} R(q)^c \in \Z[[q]]$.  Suppose $R(q) \neq 0$ and let $1 \leq b(n')\leq p-1$ be the first non-zero coefficient of $R(q)$.  Then the $cn'$-th coefficient of $\frac{1}{p} R(q)^c$ would be $\frac{b(n')^c}{p}$ which would imply that $p | b(n')$, contradicting that $R(q)$ is non-zero.  Thus $\frac{1}p f(z) \in \Z[[q]]$.  Iterating this argument over the prime factorization of $a(n_0)$ leads to the desired conclusion: $\frac1{a(n_0)} f(z) \in \Z[[q]]$.
\end{proof}

\begin{lem}\label{step 4}
Suppose $h(z) \in \Z[[q]]$ can be expressed as $h(z) = \prod_{d| N} \eta(dz)^{s_d}$ where each $s_d \in \Q$.  Then each $s_d \in \Z$--in other words, $h(z)$ is an eta-quotient.
\end{lem}

\begin{proof} Let $d_0= 1, d_1, d_2, \ldots , d_t=N$ be the divisors of $N$, and let $r = \ord_{\infty}(h(z))$.  Newton's generalized binomial theorem implies
\[ \prod_{n\geq 1} (1-q^{dn})^s = 1 - s q^d + O(q^{2d}) \, .\]
Thus we have
\[ h(z) = q^{r} \prod_{m=0}^t (1-s_{d_m}q^{d_m} + \ldots) \, .\]
We proceed by induction on the divisors of $N$.  For $d_0=1$, we see that $-s_{d_0}$ is the coefficient of $q^{r+1}$ in the Fourier expansion of $h(z)$.  Since all of $h(z)$'s coefficients are integral, we conclude that $s_{d_0}\in \Z$.  Now suppose that $s_{d_m} \in \Z$ for all $0\leq m < \ell$.  Write
\[ q^r \prod_{m=0}^{\ell-1}\prod_{n\geq1} (1-q^{d_m n})^{s_{d_m}} = \sum_{n=r}^\infty b(n)q^n \, ,\]
where each $b(n) \in \Z$.
Then the coefficient of $q^{r+d_\ell}$ in the Fourier expansion of $h(z)$ would be $b(r+d_\ell)-s_{d_\ell}$, and again we conclude that $s_{d_\ell} \in \Z$ as well.
\end{proof}

\begin{proof}[Proof of Theorem~\ref{etaquotientthm}.]  By Lemma \ref{step 2} and Corollary \ref{step 3}, $f(z)=\sum_{n=n_0}^\infty a(n)q^n$ is an integer multiple of a quotient of eta-functions to rational powers with $\frac1{a(n_0)} f(z) \in \Z[[q]]$.  Lemma \ref{step 4} then implies that $f(z)$ is in fact an eta-quotient multiplied by $a(n_0)$.
\end{proof}

%%%%%%%%%%%%%% Proof of Theorem 9
We now turn to the proof of Theorem~\ref{level-2^k-rational-cusp-sbgrp-structure}.  When $1\leq k \leq 6$, the theorem can be verified by direct computations.  Suppose $k\geq 7$.  Our first step will be to identify a minimal generating set for the group $\mathcal{M}(2^k)$.  Define the set $\{ f_{\ell,k}(z)\}_{\ell=0}^{k-1} \subset \mathcal{M}(2^k)$ as follows.  Let
\[ f_{0,k}(z) = \dfrac{\Delta(2^{k-1} z)}{\Delta(2^k z)} \, .\]
For $1\leq \ell \leq 3$, let
\[ f_{\ell,k}(z) = \begin{cases}
\dfrac{\eta(2^\ell z)^5 \eta(2^{k-1} z)^4}{\eta(2^{\ell-1}z)^2 \eta(2^{\ell+1}z)^2 \eta(2^{k-2}z) \eta(2^k z)^4 } & \text{ if } \ell \equiv k \pmod{2} \, , \\
\vspace{-0.15in}\\
\dfrac{\eta(2^\ell z)^5 \eta(2^{k-1} z) }{\eta(2^{\ell-1}z)^2 \eta(2^{\ell+1}z)^2 \eta(2^k z)^2} & \text{ if } \ell \not\equiv k \pmod2 \, . \end{cases}\]
For $4\leq \ell \leq k-2$, let
\[ f_{\ell,k}(z) = \begin{cases} \dfrac{\eta(2^{\ell} z)^2 \eta(2^{k-1} z)}{\eta(2^{\ell-1} z) \eta(2^k z)^2} & \text{ if } \ell \equiv k \pmod2 \, , \\
\vspace{-0.15in}\\
\dfrac{\eta(2^\ell z)^2 \eta(2^{k-1}z)^4 }{\eta(2^{\ell-1} z) \eta(2^{k-2} z) \eta(2^k z)^4 } & \text{ if } \ell \not\equiv k \pmod2 \, . \end{cases} \]
Finally, let
\[ f_{k-1,k}(z) = \dfrac{\eta(2^{k-1} z)^6}{\eta(2^{k-2} z)^2 \eta(2^k z)^4} \, .\]

\begin{lem}\label{level-2^k-basis-lemma}
The set $\{f_{\ell,k}\}_{\ell=0}^{k-1}$ is a minimal generating set for the multiplicative group $\mathcal{M}(2^k)$.
\end{lem}

\begin{proof} Let $g(z) = \prod_{i=0}^k \eta(2^i z)^{r_i} \in \mathcal{M}(2^k)$.  Let $i^*$ be the minimal $i$ such that $r_i$ is non-zero.  If $i^* \leq 2$, then $r_{i^*}$ must be even since $\sum_{i=i^*}^k 2^i r_i \equiv 0 \pmod{ 2^{i^*+1}}$ by \eqref{congruenceconditions}.  Suppose that $i^*=k-2$.  Then $r_{k-2}+r_{k-1}+r_{k}=0$ since $g(z)$ is weight 0.  By \eqref{congruenceconditions}, the congruence
\[ 4r_{k-2}+2r_{k-1}+r_k \equiv 0 \pmod{24} \]
must hold, which implies that $r_k$ must be even.  If $k$ is odd, then $r_{k-2}+r_{k} \equiv 0\pmod{2}$ by the third condition in \eqref{congruenceconditions}.  Thus $r_{k-2}$ must be even as well.  On the other hand, if $k$ is even, then by the third condition in \eqref{congruenceconditions}, $r_{k-1}$ must be even.  Plugging in $-r_{k-2}-r_{k-1}$ into the congruence above leads to $r_{k-1} \equiv -3 r_{k-2} \pmod{24}$.  This implies that $r_{k-2} \equiv r_{k-1} \pmod{2}$, and $r_{k-2}$ is even as well.  If $i^* = k-1$, then $r_{k-1}=-r_{k}$ since $g(z)$ is weight 0.  Combining this with the congruence condition $2 r_{k-1}+r_{k} \equiv 0 \pmod{24}$ implies that $r_{k-1}$ is a multiple of 24, and thus $g(z)$ is an integer power of $f_{0,k}(z)$.

 The proof of Lemma \ref{level-2^k-basis-lemma} now follows in this way.  For $g(z) \in \mathcal{M}(2^k)$, if $i^* < k-1$, then $r_{i^*}$ is a multiple of the power of $\eta(2^{i^*}z)$ in $f_{i^*+1,k}(z)$.  Thus for the appropriate power $c$, $f_{i^*+1,k}(z)^c g(z)= \prod_{i=i^*+1}^k \eta(2^i z)^{r'_i} \in \mathcal{M}(2^k)$.  This process continues until we are left with a function which is an integer power of $f_{0,k}(z)$.
\end{proof}

\begin{proof}[Proof of Theorem~\ref{level-2^k-rational-cusp-sbgrp-structure}.]
We continue our assumption that $k\geq 7$, and make the further restriction that $k$ be even.  The proof for odd $k$ is very similar.

For odd $j$, let the cusp $\frac{j}{2^t} \in \Q$ correspond to the
point $\left[ {{j}\atop{2^t}}\right] \in X_0(2^k)$.  Define $C :=
\left\{\left[ {{j} \atop {2^t}} \right]\right\}$ to be a complete set
of representatives of the cusps of $X_0(2^k)$.  For any divisor $d =
\sum_{\left[ {{j}\atop{2^t}}\right]\in C} n_{j,t}\left[
  {{j}\atop{2^t}}\right]\in \DIV^0(X_0(2^k)/\Q)$ whose support only
includes cusps, it follows from the proof of Lemma \ref{step 1}
that $n_{j,t} = n_{j',t}$ for all cusps $[ {{j}\atop{2^t}}], [
{{j'}\atop{2^t}}]$ of $X_0(2^k)$ that share a common denominator.  For
$1 \leq \ell \leq k-1$, $\ell \neq \frac{k}2$, define the divisors
$d_{\ell,k} = \sum_{\left[{{j}\atop{2^t}}\right]\in
  C}n_{t,\ell,k}\left[ {{j}\atop{2^t}}\right]$ as follows.  For $\ell
=1$ let
\[ n_{t,1,k} = \begin{cases} 1 & \text{ for } t=1 \\
0 & \text{ for } t=0 , 2\leq t \leq k-1 \\
-1 & \text{ for } t = k \, . \end{cases} \]
For $\ell = 2$ let
\[ n_{t,2,k} = \begin{cases} 1 & \text{ for } t=2, k-1 \\
0 & \text{ for } t=0,1 \, , \, 3\leq t \leq k-2 \\
-3 & \text{ for } t = k \, . \end{cases} \]
For $\ell = 3$ let
\[ n_{t,3,k} = \begin{cases} 1 & \text{ for } t=3 \\
0 & \text{ for } t=0,1,2 \, , \, 4 \leq t \leq k-1 \\
-4 & \text{ for } t=k \, .\end{cases} \]
For $4\leq \ell < \frac{k}2$, $\ell$ odd, let
\[ n_{t,\ell, k} = \begin{cases} 1 & \text{ for } t=\ell \\
0 & \text{ for } 0 \leq t \leq \ell-1\, , \, \ell+1 \leq t \leq k-2 \\
2^{\ell-1} & \text{ for } t = k-1 \\
-2^\ell & \text{ for } t = k \, . \end{cases} \]
For $4 \leq \ell < \frac{k}2$, $\ell$ even, let
\[ n_{t, \ell, k} = \begin{cases} 1 & \text{ for } t=\ell \\
0 & \text{ for } 0 \leq t \leq \ell-1\, , \, \ell+1 \leq t \leq k-2 \\
-2^{\ell-2} & \text{ for } t = {k-1} \\
-2^{\ell-2} & \text{ for } t = k \, . \end{cases} \]
For $\frac{k}2 < \ell \leq k-3$, $\ell$ odd, let
\[ n_{t,\ell, k} = \begin{cases} 0 & \text{ for } 0 \leq t \leq \ell-1 \\
1 & \text{ for } \ell \leq t \leq k-2 \\
1+2^{k-\ell - 1} & \text{ for } t= k-1 \\
-3\cdot 2^{k-\ell -1} + 1 & \text{ for } t = k \, . \end{cases}\]
 For $\frac{k}2 < \ell \leq k-2$, $\ell$ even, let
\[ n_{t,\ell, k} = \begin{cases} 0 & \text{ for } 0 \leq t \leq \ell-1 \\
1 & \text{ for } \ell \leq t \leq k-1 \\
-2^{k-\ell } + 1 & \text{ for } t = k \, . \end{cases}\]
For $\ell = k-1$, let
\[ n_{t,\ell, k} = \begin{cases} 0 & \text{ for } 0 \leq t \leq k-2 \\
1 & \text{ for }  t = k-1 \\
-1& \text{ for } t = k \, . \end{cases}\]

One calculates that for $1 \leq \ell \leq 3$,
\begin{equation}\label{divisor-relation-1} 2^{k-\ell - 3} d_{\ell,k} = 2^{i_{k-2\ell-1,k}}d_{\ell,k}=\Div(f_{\ell,k}) \, , \end{equation}
for $4 \leq \ell < \frac{k}2$,
\begin{equation}\label{divisor-relation-2} 2^{k-\ell -3} d_{\ell,k} = 2^{i_{k-2\ell-1,k}}d_{\ell,k}=2 \Div(f_{\ell,k}) - \Div(f_{\ell+1,k}) \, ,\end{equation}
for $\frac{k}2 < \ell \leq k-2$,
\begin{equation}\label{divisor-relation-3} 2^{\ell-4} d_{\ell,k} = 2^{i_{2\ell-k-2,k}}d_{\ell,k}=\Div(f_{\ell,k})\, ,\end{equation}
and for $\ell = k-1$,
\begin{equation}\label{divisor-relation-4} 2^{k-4} d_{k-1,k} = 2^{i_{k-2,k}}d_{\ell,k}=\Div(f_{k-1,k}) \, .\end{equation}
We note that the constants in front of the divisors $d_{\ell,k}$ in
\eqref{divisor-relation-1} to \eqref{divisor-relation-4} match the
orders of the cyclic subspaces in the statement of the theorem.  Thus
to prove the statement, it suffices to show that the $d_{\ell,k}$'s
generate all of $J_0(2^k)(\Q)_{\text{cusp}}$, and that each cyclic
subspace $\langle d_{\ell,k} \rangle$'s intersection with the subspace
$\langle d_{1,k}, \ldots , d_{\ell-1,k},d_{\ell+1,k}, \ldots \rangle$
is trivial.  Since $f_{0,k}=\Delta(2^{k-1}z)/\Delta(2^k z)$ has a
single zero at the cusp 0, without loss of generality it can be
assumed that each divisor in $J_0(2^k)(\Q)_{\text{cusp}}$ is not
supported at the cusp $[0]$.  Now to show that the $d_{\ell,k}$'s
generate all of $J_0(2^k)(\Q)_{\text{cusp}}$, it remains to show that
there exists a divisor $s = \sum_{\left[{j\atop{2^t}}\right]\in
  C}n_{t,s} \left[{j\atop{2^t}}\right]$ where $n_{t,s} = 0$ for $0\leq
t < \frac{k}2$ and $ n_{\frac{k}2,s}=1$.  One such divisor is
\[ s = \left(\sum_{4\leq \ell < \frac{k}2} 2^{k-2\ell} d_{\ell,k}\right) - \Div(f_{4,k}) \, .\]

We now show that each cyclic subspace $\langle d_{\ell,k} \rangle$'s
intersection with $\langle d_{1,k}, \ldots ,
d_{\ell-1,k},d_{\ell+1,k}, \ldots \rangle$ is trivial.  Examining the
divisors in $\langle d_{\ell,k} \rangle$, when $1\leq \ell \leq 3$ and
$\frac{k}2 < \ell \leq k-1$, we note that the support for each
multiple of $d_{\ell,k}$ does not include the cusps
$\left[{j\atop{2^t}}\right]$ for $t< \ell$ but does include the cusps
$\left[{j\atop{2^\ell}}\right]$.  For this reason they cannot be
linear combinations of the other generators.  At this point, if there
exists some non-trivial relation between the $d_{\ell,k}$'s, it must
involve only the $\ell$'s with $4 \leq \ell < k/2 $.  Since this
interval is empty if $k=8$, we let $k\geq 10$.  Suppose
$\sum_{\ell=4}^{k/2-1}c_\ell d_{\ell,k} = \Div(g)$ for some $g \in
\mathcal{M}(2^k)$. Because $d_{\ell,k}$ has order $2^{k - \ell - 3}$,
we may assume that either $c_{\ell} = 0$, or $0 \leq \ord_{2}(c_{\ell}) < k - \ell - 3$. Utilizing the relations in $\eqref{divisor-relation-1}$ to
$\eqref{divisor-relation-4}$, we see that
$g=\prod_{\ell=4}^{k/2-1}(f_{\ell,k}^2/f_{\ell+1,k})^{c_\ell/2^{k-\ell-3}}$.
When we write $g= \prod \eta(2^t z)^{r_t}$, we see for $3 \leq t \leq
k/2$ that
\[
r_t = \begin{cases} -2 \frac{c_4}{2^{k-7}} & t=3\\
\vspace{-0.15in}\\
 5\frac{c_4}{2^{k-7}}-2\frac{c_5}{2^{k-8}} & t = 4\\
\vspace{-0.15in}\\
 -2\frac{c_{t-1}}{2^{k-t-2}} + 5\frac{c_t}{2^{k-t-3}} - 2\frac{ c_{t+1}}{2^{k-t-4}} & 5 \leq t \leq \frac{k}2 - 2\\
\vspace{-0.15in}\\
 -2\frac{c_{\frac{k}2-2}}{2^{\frac{k}2-1}} +5\frac{c_{\frac{k}2-1}}{2^{\frac{k}2-2}} & t = \frac{k}2-1 \\
\vspace{-0.15in}\\
-2\frac{c_{\frac{k}{2}-1}}{2^{\frac{k}2-2}} & t ={\frac{k}{2}}\, . \end{cases}
\]
Let $t^*$ be the minimal $t$ in $4\leq t < \frac{k}{2}-1$ with $c_{t^*} \neq 0$.  Since each $r_t \in \Z$ by Lemma \ref{step 4}, examining $r_{t^*-1}$ and using $0 \leq  \ord_2(c_{t^*}) < k-t^*-3$, we see that $\ord_2(c_{t^*})= k-t^*-4$ which then implies that $\ord_2(c_{t^*+1}) = k-t^*-6$ since $r_{t^*} \in \Z$.  Repeating this argument for each successive $r_t$ with $t^*\leq t < \frac{k}2$, we see that $\ord_2(c_{t+1})= k+t^*-6-2t$.   However, this would imply that $\ord_2(r_{k/2}) = t^*-\frac{k}{2}$, which contradicts that $r_{\frac{k}2} \in \Z$.
\end{proof}

\begin{proof}[Proof of Theorem~\ref{genuszero}]
We will show that there is a bijection between weight $k$ eta-quotients
in $M_{k}(\Gamma_{0}(N))$ and cuspidal divisors of degree $\frac{k}{12} [\SL_{2}(\Z) : \Gamma_{0}(N)]$ with non-negative multiplicity at each cusp. Since
there are $\phi\left(\gcd\left(d,\frac{N}{d}\right)\right)$ cusps with denominator $d$, the claimed result follows.

It is well-known that if $f(z) \in M_{k}(\Gamma_{0}(N))$, then the number of
zeros of $f(z)$ on a fundamental domain for $\Gamma_{0}(N)$ is equal to
$\frac{k}{12} [\SL_{2}(\Z) : \Gamma_{0}(N)]$. (This follows
from Proposition 2.16 of \cite{Shimura} together with an application of
the Riemann-Hurwitz formula for the genus of $X_{0}(N)$ derived from the covering $X_{0}(N) \to X_{0}(1)$.)

If $f(z) \in M_{k}(\Gamma_{0}(N))$ is an eta-quotient, let $c_{d}
= \ord_{1/d}(f(z))$. It follows that
\[
  \sum_{d | N} c_{d} \phi\left(\gcd\left(d,\frac{N}{d}\right)\right) = \frac{k}{12} [\SL_{2}(\Z) : \Gamma_{0}(N)],
\]
and so $\text{div}(f)$ is a cuspidal divisor of the appropriate degree.

Suppose now $X_{0}(N)$ has genus zero, $h(z) \in M_{k}(\Gamma_{0}(N))$ is
an eta-quotient, and that $(c_{d} : d | N)$ is a sequence of non-negative
integers with $\sum_{d | N} c_{d} \phi\left(\gcd\left(d,\frac{N}{d}\right)\right) = \frac{k}{12} [\SL_{2}(\Z) : \Gamma_{0}(N)]$. Let $D$ be the divisor
\[
  D = \sum_{d | N} c_{d} \left(\sum_{j} c_{d} \left[ \begin{matrix} j \\ d \end{matrix} \right]\right) - \text{div}(h(z)).
\]
Here, the sum on $j$ is over all $\Gamma_{0}(N)$-classes of cusps $\left[ \begin{matrix} j \\ d \end{matrix} \right]$ with denominator $d$.

Then, $D$ is a divisor supported at cusps and has degree zero. Since $X_{0}(N)$ has genus zero, $J_{0}(N)(\Q)_{\text{cusp}}$ is trivial and every degree zero cuspidal divisor is the divisor of some modular function $i(z)$ with rational Fourier coefficients, and with leading Fourier coefficient $1$. Let
$f(z) = h(z) i(z)$. If $C$ is the least common multiple of the denominators
of the coefficients of $f(z)$, then $C f(z) \in M_{k}(\Gamma_{0}(N)) \cap \Z[[q]]$ and is nonvanishing on $\mathbb{H}$. Corollary~\ref{step 3} implies that
every coefficient is a multiple of $C$ and therefore $C = 1$. Theorem~\ref{etaquotientthm} now implies that $f(z)$ is an eta-quotient, and we have that
\[
  \text{div}(f) = \sum_{d | N} c_{d} \left(\sum_{j} \left[ \begin{matrix} j \\ d \end{matrix} \right]\right).
\]
\end{proof}

\bibliographystyle{amsplain}
\bibliography{etarefs}

\end{document}